\documentclass[11pt]{amsart}
\usepackage[left=1.25in,right=1.25in]{geometry}
\usepackage{amsthm}
\usepackage{amsmath}
\usepackage{mathtools}
\usepackage{amssymb}
\usepackage{hyperref}
\usepackage{cleveref}
\usepackage{tikz-cd}
\usepackage{tikz}
\usepackage{array}
\usetikzlibrary{decorations.markings}
\usepackage{enumerate}
\usepackage{amsfonts}
\usepackage{todonotes}
\usepackage[utf8]{inputenc}
\usepackage{ytableau}
\usepackage{subcaption}
\usepackage[export]{adjustbox}
\usepackage{mathdots}
\usepackage{comment}
\usepackage{rotating}
\usetikzlibrary{shapes}

\newtheorem{theorem}{Theorem}[section]
\newtheorem{prop}[theorem]{Proposition}
\newtheorem{conj}[theorem]{Conjecture}
\newtheorem{lemma}[theorem]{Lemma}
\newtheorem{cor}[theorem]{Corollary}

\theoremstyle{definition}

\newtheorem{defin}[theorem]{Definition}

\renewcommand{\S}{\mathfrak{S}}
\DeclareMathOperator{\BPD}{BPD}
\newcommand{\wt}{\mathrm{wt}}
\newcommand{\w}{\mathrm{w}}
\renewcommand{\empty}{\mathrm{empty}}

\tikzset{
path/.style = {line width = 2pt, color = blue, dashed}}
\tikzset{
zero_node/.style = {draw, color = orange, line width = 2pt, cross out}}

\makeatother

\title[Principal specializations of Schubert polynomials]{Principal specializations of Schubert polynomials, multi-layered permutations and asymptotics}

\author{Ningxin Zhang}
\address{School of Mathematical Sciences, Peking University}
\email{\href{mailto:xinxin9909@pku.edu.cn}{{\tt xinxin9909@pku.edu.cn}}}

\date{\today}

\begin{document}

\begin{abstract}
Let $v(n)$ be the largest principal specialization of Schubert polynomials for layered permutations
$v(n) := \max_{w \in \mathcal{L}_n} \S_w(1,\ldots,1)$.
Morales, Pak and Panova proved that there is a limit 
\[\lim_{n \to \infty} \frac{\log v(n)}{n^2},\]
and gave a precise description of layered permutations reaching the maximum. In this paper, we extend Morales Pak and Panova's results to generalized principal specialization $\S_w(1,q,q^2,\ldots)$ for multi-layered permutations when $q$ equals a root of unity. 
\end{abstract}
\maketitle

\section{Introduction}\label{sec:intro}
The study of principal specializations $\Upsilon_w := \S_w(1,\ldots,1)$ of Schubert polynomials has attracted a lot of interests recently \cite{dennin2022pattern, gao2021principal, meszaros2021inclusion}. Geometrically, $\Upsilon_w$ equals the degree of the matrix Schubert variety of $w$ and combinatorially, it equals the number of pipe dreams or bumpless pipe dreams of $w$. In general, principal specializations $\S_w(1,q,q^2,\ldots)$ can be calculated by Macdonald's identity \cite[(6.11$_q$)]{SchubertNotes}
\[\S_w(1,q,q^2,\ldots) = \sum_{a = (a_1,\ldots,a_\ell) \in Re(w)} q^{\phi(a)} \frac{(1-q^{a_1})\cdots(1-q^{a_\ell})}{(1-q)\cdots(1-q^\ell)},\]
where $\ell = \ell(w)$ is the length of $w$, $Re(w)$ is the set of reduced words of $w \in S_n$ and $\phi(a)$ is defined as $\phi(a) := \sum_{a_i < a_{i+1}} i$.
It was first proved by Fomin and Stanley \cite{Fomin} in an algebraic way. Recently, Billey, Holroyd and Young \cite{Billey} found a bijective proof of this identity. 

Let $u(n)$ be the largest principal specialization
$u(n) := \max_{w \in S_n} \Upsilon_w$.
\begin{conj}[\cite{stanley}]\label{conj:Stanley}
    There is a limit \ $\lim_{n \to \infty} \frac{\log u(n)}{n^2}$.
\end{conj}
Stanley \cite{stanley} made the conjecture above in 2017. Moreover, he asked about a description of such permutations $w \in S_n$ that achieve the maximum.

Layered permutations $w(b_1,\ldots,b_k)$ are defined as 
\[w(b_1,\ldots,b_k) := (b_1,\ldots,1,b_1+b_2,\ldots,b_1+1,\ldots,n,\ldots n-b_k+1),
\] where $n = b_1+\cdots+b_k$. Denote by $\mathcal{L}_n$ the set of layered permutations in $S_n$. 
For example, $w(1,n-1) = (1,n,n-1,\ldots,2)$. It is well known that 
$\Upsilon_{(1,n,n-1,\ldots,2)} = C_{n-1}$, where $C_k=\frac{(2k)!}{k!(k+1)!}$ is the $k$-th Catalan number. 

Let $v(n)$ be the largest principal specialization for layered permutations 
\[v(n):= \max_{w \in \mathcal{L}_n} \Upsilon_w.\]
Morales, Pak and Panova \cite{MPP} established the asymptotic behavior of principal specializations for layered permutations and resolved Stanley's conjecture partially.
\begin{theorem}[\cite{MPP}]
    The limit \[\lim_{n \to \infty} \frac{\log_2 v(n)}{n^2} 
 = \frac{\gamma}{\log 2} \approx 0.2932\] exists, where $\gamma \approx 0.2032558981$ is a universal constant. Moreover, the maximum $v(n)$ is achieved at such layered permutations
 \[w(\ldots,b_2,b_1), \text{ where } b_i \sim \alpha^{i-1} (1-\alpha)n \text{ as } n \to \infty,\]
 for every $i$, where $\alpha \approx 0.4331818312$ is a universal constant. 
\end{theorem}

Merzon and Smirnov \cite{Merzon} made the following conjecture, which lead to \Cref{conj:Stanley}. 
\begin{conj}[\cite{Merzon}]\label{conj:Merzon}
    For every $n$, all permutations reaching the maximum $u(n)$ are layered permutations. In other words, $u(n) = v(n)$. 
\end{conj}

In this paper, we extend the above results to generalized principal specializations of Schubert polynomials $\S_w (1,q,q^2,...)$, where $q$ is a root of unity. To simplify the writing, we focus on the principal specializations 
\[\Phi_w :=|\S_w(1,q,q^2,\ldots)|_{q=-1} = \left|\S_w(1,-1,\ldots)\right|.\] 

\begin{defin}\label{def:doubly}
We define \emph{doubly layered permutations} $w_2(2b_1,\ldots,2b_{k-1},2b_k+r)$ as
\[\begin{aligned}  w_2(2b_1,\ldots,2b_{k-1},2b_k+r)  := &(2b_1{-}1,2b_1,\ldots,1,2,2b_1{+}2b_2{-}1, 2b_1+2b_2, \ldots,2b_1+1,2b_1+2, \\ & \ldots, n-1,n, n-3,n-2,\ldots) \in S_n
\end{aligned}\]
where $n = 2b_1 + \ldots + 2b_k +r$ and $r = 0$ or $1$ is the remainder of $n\bmod2$.
\end{defin}
See \Cref{fig:double layer} for examples of such permutation matrices. Denote by $\mathcal{DL}_n$ the set of doubly layered permutations in $S_n$. 
\begin{figure}[h]
    \centering
    \begin{subfigure}{.45\textwidth}
    \centering
        \begin{tikzpicture}
        \draw (0.1,0.1) rectangle (3.3,3.3);
        \filldraw[fill = blue!20] (0.1,3.3) rectangle (0.5,2.9) (0.5,2.9) rectangle (1.7,1.7) (1.7,1.7) rectangle (3.3,0.1);
        \foreach \i in {0.2,0.4}\filldraw (\i,3.4-\i) circle (.05);
        \foreach \i in {0.6,1,1.4} \filldraw (\i,\i+1.4) circle (.05) (\i+0.2,\i+1.2) circle (.05);
        \foreach \i in {1.8,2.2,2.6,3.0} \filldraw (\i,\i-1.4) circle (.05) (\i+0.2,\i-1.6) circle (.05);
        \end{tikzpicture}
        \caption{when $n$ is even}
        \label{subfig:layer even}
    \end{subfigure}
    \begin{subfigure}{.45\textwidth}
    \centering
        \begin{tikzpicture}
        \draw (-0.1,-0.1) rectangle (3.3,3.3);
        \filldraw[fill = blue!20] (-0.1,3.3) rectangle (0.3,2.9) (0.3,2.9) rectangle (1.5,1.7) (1.5,1.7) rectangle (3.3,-0.1);
        \foreach \i in {0,0.2}\filldraw (\i,3.2-\i) circle (.05);
        \foreach \i in {0.4,0.8,1.2} \filldraw (\i,\i+1.6) circle (.05) (\i+0.2,\i+1.4) circle (.05);
        \foreach \i in {1.8,2.2,2.6,3} \filldraw (\i,\i-1.4) circle (.05) (\i+0.2,\i-1.6) circle (.05);
        \filldraw (1.6,0) circle (.05);
        \end{tikzpicture}
        \caption{when $n$ is odd}
        \label{subfig:layer odd}
    \end{subfigure}
    \caption{Doubly layered permutations}
    \label{fig:double layer}
\end{figure}
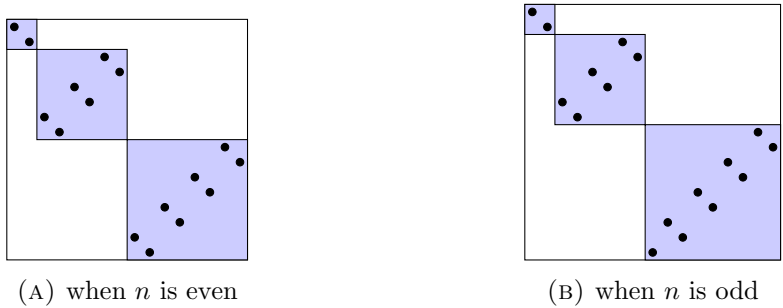

Some doubly layered permutations achieve nice values under the principal specialization at $q=-1$, analogous to the behavior of certain layered permutations at $q=1$.
\begin{theorem}\label{thm:Catalan for double layer}
For $w = w_2(2,n-2) = (1,2,n-1,n,n-3,n-2,\ldots)$,
\[\Phi_{w}=\begin{cases}
C_{k-1}^2 & \text{ if }n=2k\text{ is even},\\
C_{k-1}C_k & \text{ if }n=2k+1\text{ is odd},
\end{cases}\]
where $C_k$ is the $k$-th Catalan number.
\end{theorem}

Let $\Tilde{u}(n)$, $\Tilde{v}(n)$ be the largest principal specialization at $q = -1$ \[\begin{aligned}\Tilde{u}(n) &:= \max_{w \in S_n} \Phi_w, \\ \Tilde{v}(n) &:= \max_{w \in \mathcal{DL}_n} \Phi_w. \end{aligned} \]
Analogous to Morales, Pak and Panova's work, we have the following asymptotic result. 

\begin{theorem}\label{thm:asymptotic}
    There is a limit \[\lim_{n \to \infty} \frac{\log_2 \Tilde{v}(n)}{n^2} = \frac{1}{2}\lim_{n \to \infty} \frac{\log_2 v(n)}{n^2}  = \frac{\gamma}{2\log2}\approx 0.1466.\]
\end{theorem}

Furthermore, as \Cref{conj:Stanley} and \Cref{conj:Merzon}, we naturally conjecture that
\begin{conj}
    For every $n$, all permutations reaching the maximum $\Tilde{u}(n)$ are doubly layered permutations. Thus, there is a limit $\lim_{n \to \infty} \frac{\log_2 \Tilde{u}(n)}{n^2}$.
\end{conj}

The outline of the paper is as follows. In \Cref{sec:prelim} we give necessary background on bumpless pipe dreams and Lindstr\"om-Gessel-Viennot lemma. In \Cref{sec:main part} we prove our main results, which are \Cref{thm:Catalan for double layer} and \Cref{thm:asymptotic}. In \Cref{sec:multi} we show general results of principal specializations at roots of unity for multi-layered permutations.

\section{Preliminaries}\label{sec:prelim}
We write a permutation $w \in S_n$ in its one-line notation $w = w_1 w_2 ... w_n$. Given $u \in S_m$ and $v \in S_n$, we define the following permutation: \[u \times v  := (u_1,...,u_m,v_1+m,...,v_n+m) \in S_{m+n}.\] In particular, $1^m \times v = (1,...,m,v_1+m,...,v_n+m)$.
A property of Schubert polynomial (e.g., see \cite[(4.6)]{SchubertNotes}) shows that 
\begin{equation}\label{equ:prop}
\mathfrak{S}_{u \times v} = \mathfrak{S}_u \cdot \mathfrak{S}_{1^m \times v}.
\end{equation}

\subsection{Bumpless pipe dreams}\label{subsec:bumpless}
Bumpless pipe dreams were first introduced by Lam, Lee and Shimozono \cite{lam-lee-shimo} in their study of back stable Schubert calculus.
\begin{defin}
A bumpless pipe dream (BPD) of $w \in S_n$ is a filling of $[n] \times [n]$ by 
\begin{center}\includegraphics[scale=0.4]{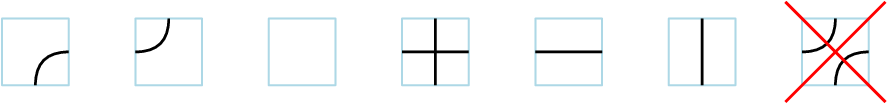}\end{center}
such that the pipe which starts at column $i$ ends at row $w(i)$ for any $i >0$, and that every pair of pipes crosses at most once.
\end{defin}
Denote the set of bumpless pipe dreams for $w$ as $\BPD(w)$. And for a BPD $D$ of $w$, let its set of empty tiles be $\empty(D)$. The weight of a bumpless pipe dream $D$ is defined as \[\wt(D) = \prod_{(i,j) \in \empty(D)} x_i.\]
Lam, Lee and Shimozono \cite{lam-lee-shimo} showed that Schubert polynomials can be calculated by the weight of bumpless pipe dreams.

\begin{theorem}[\cite{lam-lee-shimo}]\label{thm:BPD wt}
Given $n > 0$, for any $w \in S_n$, \[\mathfrak{S}_w = \sum_{D \in \BPD(w)} \wt(D).\]
\end{theorem}

\subsection{The Lindstr\"om-Gessel-Viennot lemma}\label{subsec:Lindstrom}
The Lindstr\"om-Gessel-Viennot lemma is a useful method to compute the weighted sum of non-intersecting lattice paths. See Section 2 of \cite{ec1} for details. 

Let $G$ be an acyclic weighted directed graph. For an edge $e$, let $\wt_e$ be its weight and for a directed path $P$, let its weight $\wt(P)$ be the product of the weights of its edges. For any pair of vertices $x$ and $y$, we denote by $\wt(x,y)$ the sum of weights of all paths from $x$ to $y$. Now consider two sets of points $S = \{s_1,...,s_n\}$ and $T = \{t_1,...,t_n\}$ in $G$. Define 
\[
M_{S,T}:=\left(\begin{array}{cccc}
\wt(s_1, t_1) & \wt(s_1, t_2) & \cdots & \wt(s_1, t_n) \\
\wt(s_2, t_1) & \wt(s_2, t_2) & \cdots & \wt(s_2, t_n) \\
\vdots & \vdots & \ddots & \vdots \\
\wt(s_n, t_1) & \wt(s_n, t_2) & \cdots & \wt(s_n, t_n)
\end{array}\right).
\]
An $n$-path from $S$ to $T$ means an $n$-tuple $P = (P_1,\ldots,P_n)$ of paths with each $P_i$ goes from $s_i$ to $t_{\sigma(i)}$, where $\sigma \in S_n$ is a permutation. Denote by $\sigma(P)$ the permutation $\sigma$ as above. An $n$-path is \emph{non-intersecting} if no paths share a common vertex. The Lindstr\"om-Gessel-Viennot lemma states that the determinant of $M_{S,T}$ equals the signed sum over all $n$-tuples $P$ of non-intersecting paths from $S$ to $T$.

\begin{lemma}\label{lemma:Lindstrom}
\[\det(M) = \sum_{
\substack{P = (P_1, \ldots, P_n): S \rightarrow T\\
\mathrm{non{-}intersecting}}}
\operatorname{sign}(\sigma(P)) \prod_{i=1}^n\w(P_i).\]
\end{lemma}

\section{Principal specializations for doubly layered permutations}\label{sec:main part}
In this section, we prove our main results \Cref{thm:Catalan for double layer} and \Cref{thm:asymptotic}. The idea is to compare $\Phi_w$ for doubly layered permutations with $\Upsilon_w$ for layered permutations. Principal specializations can be computed via BPDs by \Cref{thm:BPD wt}. We will construct certain weighted graphs and convert BPDs to $n$-paths. From \Cref{lemma:Lindstrom}, weighted sum of $n$-paths can be represented by a matrix determinant. 

\begin{defin}\label{def:gen stair}
Given $n = m+p$, where $m,p \in \mathbb{Z}_{>0}$, we define a \emph{generalized staircase partition} $\rho_n^m$ as \[\rho _n^m = \rho_{m+p}^m = (\underbrace{m+p,\ldots,m+p}_m,m+p-1,\ldots, m).\]
In particular, when $m = 1$, $\rho_n^1 = \rho_n := (n,n-1,\ldots,1)$ is the normal staircase partition.
\end{defin}
\begin{defin}\label{def:stair graph}
    Let $\rho$ be a generalized staircase partition, we construct a directed graph $G_{\rho}$ as follows.
    \begin{enumerate}
        \item Put a vertex in each box of the Young diagram $\rho$.
        \item Build an edge between every pair of adjacent boxes upwards and rightwards.
        \item Let each edge have weight $1$.
    \end{enumerate}
\end{defin}
See \Cref{fig:layer} for example of $\rho_5$ and $G_{\rho_5}$ where each green edge is of weight $1$.
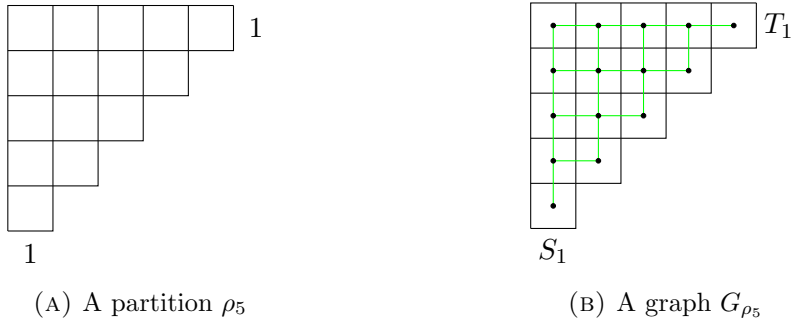
\begin{figure}[h]
    \centering
    \begin{subfigure}{.45\textwidth}
    \centering
    \begin{tikzpicture}[scale = 0.6]
        \foreach \i in {0,1,2,3,4} \draw (\i,\i)--(\i+1,\i)--(\i+1,\i+1);
        \draw (0,0)--(0,5)--(5,5);
        \draw (1,1)--(1,5) (2,2)--(2,5) (3,3)--(3,5) (4,4)--(4,5);
        \draw (0,4)--(4,4) (0,3)--(3,3) (0,2)--(2,2) (0,1)--(1,1);
        \node at (0.5,-0.5) {$1$}; \node at (5.5,4.5) {$1$};
    \end{tikzpicture}
    \caption{A partition $\rho_5$}
    \label{subfig:layer_BPD}
    \end{subfigure}
    \begin{subfigure}{.45\textwidth}
    \centering
    \begin{tikzpicture}[scale = 0.6]
        \foreach \i in {0,1,2,3,4} \draw (\i,\i)--(\i+1,\i)--(\i+1,\i+1);
        \draw (0,0)--(0,5)--(5,5);
        \draw (1,1)--(1,5) (2,2)--(2,5) (3,3)--(3,5) (4,4)--(4,5);
        \draw (0,4)--(4,4) (0,3)--(3,3) (0,2)--(2,2) (0,1)--(1,1);
        \node at (0.5,-0.5) {$S_1$}; \node at (5.5,4.5) {$T_1$};
        \foreach \i in {0.5,1.5,2.5,3.5} \draw [color = green] (\i,\i)--(\i,4.5) (0.5,\i+1)--(\i+1,\i+1);
        \foreach \i in {0.5,1.5,2.5,3.5,4.5} \filldraw (0.5,\i) circle (.05);
        \foreach \i in {1.5,2.5,3.5,4.5} \filldraw (1.5,\i) circle (.05);
        \foreach \i in {2.5,3.5,4.5} \filldraw (2.5,\i) circle (.05);
        \filldraw (3.5,3.5) circle (.05) (3.5,4.5) circle (.05) (4.5,4.5) circle (.05);
    \end{tikzpicture}
    \caption{A graph $G_{\rho_5}$}
    \label{subfig:layer_graph}
    \end{subfigure}
    \caption{A staircase partition and its corresponding graph}
    \label{fig:layer}
\end{figure}

\begin{defin}\label{def:gen double stair}
    Given $n = 2m+2p+r$ where $r = 0$ or $1$ is the remainder of $n \bmod 2$, we define a \emph{generalized double staircase partition} $\Tilde{\rho}_{n}^{2m}$ as
    \[ \Tilde{\rho}_{n}^{2m} := 
    \begin{cases}
         (\underbrace{n,\ldots,n}_{2m},n-2,n-2,\ldots,2m,2m) & \text{ if } n \text{ is even}, \\
         (\underbrace{n,\ldots,n}_{2m},n-2,n-2,\ldots,2m+1,2m+1,2m) & \text{ if } n \text{ is odd}.
    \end{cases}\]
In prticular, we write $\Tilde{\rho}_n := \Tilde{\rho}_n^2 = \begin{cases} (n,n,n-2,n-2,\ldots,2,2) & \text{ if } n \text{ is even} \\ (n,n,n-2,n-2,\ldots,3,3,2) & \text{ if } n \text{ is odd} \end{cases}$ for simplicity when $m = 1$.
\end{defin}
\begin{defin}\label{def:double stair graph}
Let $\Tilde{\rho}$ be a generalized double staircase partition. We construct a directed graph $G_{\Tilde{\rho}}$ as follows.
\begin{enumerate}
    \item Put a vertex in each box of the Young diagram $\Tilde{\rho}$.
    \item Build an edge between every pair of adjacent boxes upwards and rightwards.
    \item Let each vertical edge have weight $1$ and each horizontal edge on row $j$ have weight $(-1)^{j-1}$.
\end{enumerate}
\end{defin}
See \Cref{subfig:c} where each green edge has weight $1$ and red edge has weight $-1$. 

In a lattice graph, we denote by $(i,j)$ the vertex in the $i$-th row from top to bottom, in the $j$-th column from left to right.  
\begin{figure}[h]
\centering
\begin{subfigure}{.45\textwidth}
    \centering
    \begin{tikzpicture}[scale = 0.6]
        \draw (0,0) grid (6,6);
        \draw [line width = 2pt] (0,6)--(0,0)--(1,0)--(2,0)--(2,1)--(2,2)--(3,2)--(4,2)--(4,3)--(4,4)--(5,4)--(6,4)--(6,5)--(6,6)--(0,6);
        \node (s_1) at (0.5,-0.5) {1};\node (t_1) at (6.5,5.5) {1};
        \node (s_2) at (1.5,-0.5) {2};\node (t_2) at (6.5,4.5) {2};
        \node (s_3) at (2.5,-0.5) {3};\node (t_3) at (6.5,3.5) {5};
        \node (s_4) at (3.5,-0.5) {4};\node (t_4) at (6.5,2.5) {6};
        \node (s_5) at (4.5,-0.5) {5};\node (t_5) at (6.5,1.5) {3};
        \node (s_6) at (5.5,-0.5) {6};\node (t_6) at (6.5,0.5) {4};
        \node (1) at (1.5,4.5) {-1};\node (2) at (2.5,4.5) {-1};\node (3) at (1.5,3.5) {1};\node (4) at (2.5,3.5) {1};
        \draw [path] (1.5,0)--(1.5,2.5)--(3.5,2.5)--(3.5,4.5)--(6,4.5);
        \draw [path] (0.5,0)--(0.5,5.5)--(6,5.5);
        \draw [path] (2.5,0)--(2.5,1.5)--(6,1.5);
        \draw [path] (3.5,0)--(3.5,0.5)--(6,0.5);
        \draw [path] (4.5,0)--(4.5,3.5)--(6,3.5);
        \draw [path] (5.5,0)--(5.5,2.5)--(6,2.5);
    \end{tikzpicture}
    \caption{}
    \label{subfig:a}
\end{subfigure}
\begin{subfigure}{.45\textwidth}
    \centering
    \begin{tikzpicture}[scale = 0.6]
        \draw [line width = 2pt] (0,6)--(0,0)--(1,0)--(2,0)--(2,1)--(2,2)--(3,2)--(4,2)--(4,3)--(4,4)--(5,4)--(6,4)--(6,5)--(6,6)--(0,6);
        \draw (1,0)--(1,6);\draw (0,1)--(2,1);
        \draw (2,2)--(2,6);\draw (0,2)--(2,2);
        \draw (3,2)--(3,6);\draw (0,3)--(4,3);
        \draw (4,4)--(4,6);\draw (0,4)--(4,4);
        \draw (5,4)--(5,6);\draw (0,5)--(6,5);
        \node (s_1) at (0.5,-0.5) {1};
        \node (s_2) at (1.5,-0.5) {2};
        \node (t_1) at (6.5,5.5) {1};
        \node (t_2) at (6.5,4.5) {2};
        \node (1) at (1.5,4.5) {-1};\node (2) at (2.5,4.5) {-1};\node (3) at (1.5,3.5) {1};\node (4) at (2.5,3.5) {1};
        \draw [path] (1.5,0)--(1.5,2.5)--(3.5,2.5)--(3.5,4.5)--(6,4.5);
        \draw [path] (0.5,0)--(0.5,5.5)--(6,5.5);
    \end{tikzpicture}
    \caption{}
    \label{subfig:b}
\end{subfigure}
\begin{subfigure}{.45\textwidth}
    \centering
    \begin{tikzpicture}[scale = 0.6]
        \draw [line width = 2pt] (0,6)--(0,0)--(1,0)--(2,0)--(2,1)--(2,2)--(3,2)--(4,2)--(4,3)--(4,4)--(5,4)--(6,4)--(6,5)--(6,6)--(0,6);
        \draw (1,0)--(1,6);\draw (0,1)--(2,1);
        \draw (2,2)--(2,6);\draw (0,2)--(2,2);
        \draw (3,2)--(3,6);\draw (0,3)--(4,3);
        \draw (4,4)--(4,6);\draw (0,4)--(4,4);
        \draw (5,4)--(5,6);\draw (0,5)--(6,5);
        \foreach \i in {0.5,1.5,2.5,3.5,4.5,5.5} \filldraw (0.5,\i) circle (.05) (1.5,\i) circle (.05);
        \foreach \i in {2.5,3.5,4.5,5.5} \filldraw (2.5,\i) circle (.05) (3.5,\i) circle (.05);
        \foreach \i in {4.5,5.5} \filldraw (4.5,\i) circle (.05) (5.5,\i) circle (.05);
        \draw [color = green] (0.5,0.5)--(0.5,5.5) (1.5,0.5)--(1.5,5.5) (2.5,2.5)--(2.5,5.5) (3.5,2.5)--(3.5,5.5) (4.5,4.5)--(4.5,5.5) (5.5,4.5)--(5.5,5.5);
        \draw [color = green] (0.5,5.5)--(5.5,5.5) (0.5,3.5)--(3.5,3.5) (0.5,1.5)--(1.5,1.5);
        \draw [color = red] (0.5,4.5)--(5.5,4.5) (0.5,2.5)--(3.5,2.5) (0.5,0.5)--(1.5,0.5);
        \node (1) at (6.5,5.5) {$T_1$};
        \node (2) at (6.5,4.5) {$T_2$};
        \node (3) at (0.5,-0.5) {$S_1$};
        \node (4) at (1.5,-0.5) {$S_2$};
    \end{tikzpicture}
    \caption{}
    \label{subfig:c}
\end{subfigure}
\begin{subfigure}{.45\textwidth}
    \centering
    \begin{tikzpicture}[scale = 0.6]
        \draw [line width = 2pt] (0,6)--(0,0)--(1,0)--(2,0)--(2,1)--(2,2)--(3,2)--(4,2)--(4,3)--(4,4)--(5,4)--(6,4)--(6,5)--(6,6)--(0,6);
        \draw (1,0)--(1,6);\draw (0,1)--(2,1);
        \draw (2,2)--(2,6);\draw (0,2)--(2,2);
        \draw (3,2)--(3,6);\draw (0,3)--(4,3);
        \draw (4,4)--(4,6);\draw (0,4)--(4,4);
        \draw (5,4)--(5,6);\draw (0,5)--(6,5);
        \foreach \i in {0.5,1.5,2.5,3.5,4.5,5.5} \filldraw (0.5,\i) circle (.05) (1.5,\i) circle (.05);
        \foreach \i in {2.5,3.5,4.5,5.5} \filldraw (2.5,\i) circle (.05) (3.5,\i) circle (.05);
        \foreach \i in {4.5,5.5} \filldraw (4.5,\i) circle (.05) (5.5,\i) circle (.05);
        \draw [color = green] (0.5,0.5)--(0.5,5.5) (1.5,0.5)--(1.5,5.5) (2.5,2.5)--(2.5,5.5) (3.5,2.5)--(3.5,5.5) (4.5,4.5)--(4.5,5.5) (5.5,4.5)--(5.5,5.5);
        \draw [color = green] (0.5,5.5)--(5.5,5.5) (0.5,3.5)--(3.5,3.5) (0.5,1.5)--(1.5,1.5);
        \draw [color = red] (0.5,4.5)--(5.5,4.5) (0.5,2.5)--(3.5,2.5) (0.5,0.5)--(1.5,0.5);
        \draw [dashed, color = blue, line width = 2pt] (1.5,0.5)--(1.5,2.5)--(3.5,2.5)--(3.5,4.5)--(5.5,4.5);
        \draw [dashed, color = blue, line width = 2pt] (0.5,0.5)--(0.5,5.5)--(5.5,5.5);
        \node (1) at (6.5,5.5) {$T_1$};
        \node (2) at (6.5,4.5) {$T_2$};
        \node (3) at (0.5,-0.5) {$S_1$};
        \node (4) at (1.5,-0.5) {$S_2$};
    \end{tikzpicture}
    \caption{}
    \label{subfig:d}
\end{subfigure}
\caption{A BPD of $w = (1,2,5,6,3,4)$}
\label{fig:BPD-eg}
\end{figure}
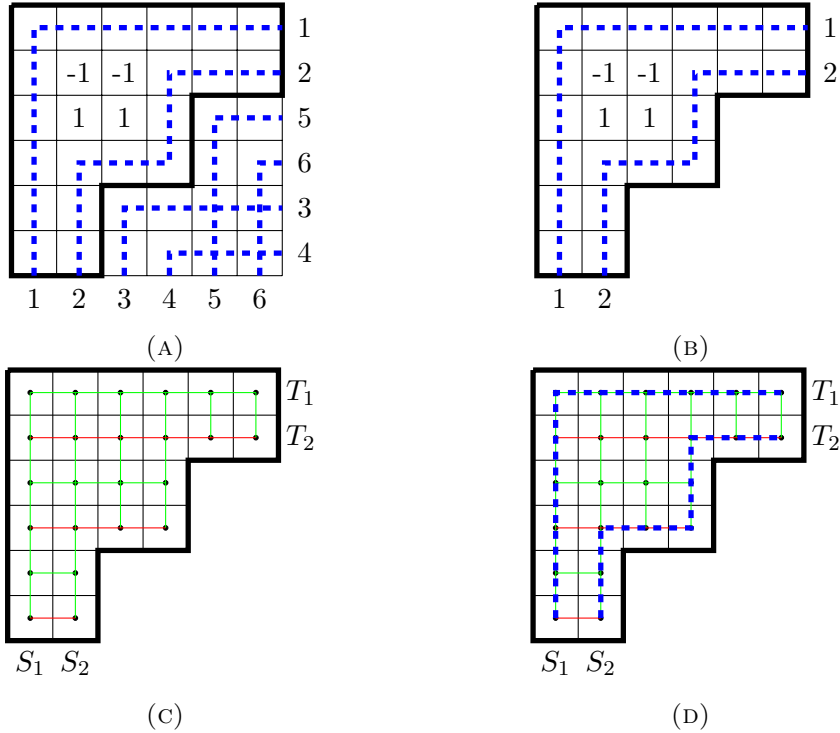
\subsection{Catalan numbers in doubly layered permutation} 
Catalan numbers show up in natural ways in the principal specialization of interest. In this subsection, we construct a graph $\Tilde{G}$ corresponding to a double staircase partition $\Tilde{\rho}$. Then, for a fixed doubly layered permutation $w$, each bumpless pipe dream of $w$ in $\Tilde{\rho}$ corresponds to a $2$-path in the graph $\Tilde{G}$. We calculate the principal specialization to prove \Cref{thm:Catalan for double layer} by converting the weighted sum of BPDs to the weighted sum of $2$-paths and applying \Cref{thm:BPD wt} and \Cref{lemma:Lindstrom}.

Now we focus on $\Phi_w$ for doubly layered permutation $w_2(2,n-2)$. For a $\BPD$ $D$ of $w_2(2,2k-2,r)$, pipe $i$ is fixed for any $i \ge 3$. The remaining boxes form a partition $\Tilde{\rho}_n$. Pipes $1$ and $2$ in $\Tilde{\rho}_n$ naturally form a pair of non-intersecting paths $P_1$ and $P_2$ in $G_{\Tilde{\rho}}$, see \Cref{fig:BPD-eg} for an example. Denote by $P_{G_{\rho}}$ the weighted sum all over non-intersecting 2-paths $P=(P_1,P_2)$ from $S = \{S_1,S_2\}$ to $T=\{T_1,T_2\}$. We have the follow lemma.

\begin{lemma}\label{lemma:weight-equ}
    For each bumpless pipe dream D of $w_2(2,n-2)$ and its corresponding directed $2$-path $P=(P_1, P_2)$ in $\Tilde{G}_{\Tilde{\rho}}$, we have \[\wt(D) = \delta(n) \cdot \wt(P),\] where $\wt(P) = \wt(P_1)\wt(P_2)$ and $\delta(n) = \begin{cases}
        -1 & n \equiv 3 \mod 4, \\
         1 & otherwise. \end{cases}$
\end{lemma}
\begin{proof}
First, we consider the \emph{Rothe} BPD $D_0$ in which pipe $1$ and pipe $2$ both go straight up and straight right and its corresponding $2$-path $P_0$ (See \Cref{fig:certain BPD}). It is not difficult to verify the equality above for $n = 2k$, $n = 4k+1$ and $n = 4k+3$, respectively. 
\begin{figure}[h]
\centering
\begin{subfigure}{.45\textwidth}
    \centering
    \begin{tikzpicture}[scale = 0.4]
        \draw (0,6)--(0,0)--(1,0)--(2,0)--(2,1)--(2,2)--(3,2)--(4,2)--(4,3)--(4,4)--(5,4)--(6,4)--(6,5)--(6,6)--(0,6);
        \draw (1,0)--(1,6);\draw (0,1)--(2,1);
        \draw (2,2)--(2,6);\draw (0,2)--(2,2);
        \draw (3,2)--(3,6);\draw (0,3)--(4,3);
        \draw (4,4)--(4,6);\draw (0,4)--(4,4);
        \draw (5,4)--(5,6);\draw (0,5)--(6,5);
        \draw [path] (1.5,0)--(1.5,4.5)--(6,4.5);
        \draw [path] (0.5,0)--(0.5,5.5)--(6,5.5);
    \end{tikzpicture}
    \caption{BPD $D_0$}
    \label{subfig:D_0}
\end{subfigure}
\begin{subfigure}{.45\textwidth}
    \centering
    \begin{tikzpicture}[scale = 0.4]
        \draw (0,6)--(0,0)--(1,0)--(2,0)--(2,1)--(2,2)--(3,2)--(4,2)--(4,3)--(4,4)--(5,4)--(6,4)--(6,5)--(6,6)--(0,6);
        \draw (1,0)--(1,6);\draw (0,1)--(2,1);
        \draw (2,2)--(2,6);\draw (0,2)--(2,2);
        \draw (3,2)--(3,6);\draw (0,3)--(4,3);
        \draw (4,4)--(4,6);\draw (0,4)--(4,4);
        \draw (5,4)--(5,6);\draw (0,5)--(6,5);
        \foreach \i in {0.5,1.5,2.5,3.5,4.5,5.5} \filldraw (0.5,\i) circle (.05) (1.5,\i) circle (.05);
        \foreach \i in {2.5,3.5,4.5,5.5} \filldraw (2.5,\i) circle (.05) (3.5,\i) circle (.05);
        \foreach \i in {4.5,5.5} \filldraw (4.5,\i) circle (.05) (5.5,\i) circle (.05);
        \draw [color = green] (0.5,0.5)--(0.5,5.5) (1.5,0.5)--(1.5,5.5) (2.5,2.5)--(2.5,5.5) (3.5,2.5)--(3.5,5.5) (4.5,4.5)--(4.5,5.5) (5.5,4.5)--(5.5,5.5);
        \draw [color = green] (0.5,5.5)--(5.5,5.5) (0.5,3.5)--(3.5,3.5) (0.5,1.5)--(1.5,1.5);
        \draw [color = red] (0.5,4.5)--(5.5,4.5) (0.5,2.5)--(3.5,2.5) (0.5,0.5)--(1.5,0.5);
        \draw [dashed, color = blue, line width = 2pt] (1.5,0.5)--(1.5,4.5)--(5.5,4.5);
        \draw [dashed, color = blue, line width = 2pt] (0.5,0.5)--(0.5,5.5)--(5.5,5.5);
    \end{tikzpicture}
    \caption{$2$-path $P_0$}
    \label{subfig:2-path}
\end{subfigure}
\caption{BPD $D_0$ and its corresponding $2$-path}
\label{fig:certain BPD}
\end{figure}
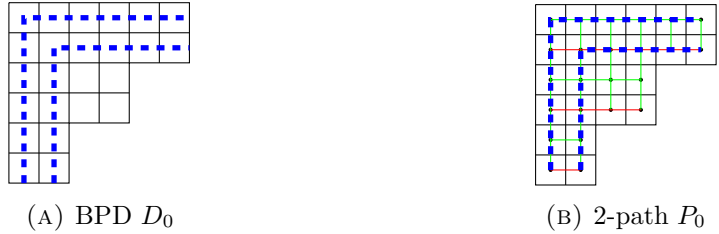

Next, recall that a \emph{simple droop} on a BPD is a local move that turns a part of pipe route $(i,j)\rightarrow(i-1,j)\rightarrow(i-1,j+1)$ to $(i,j)\rightarrow(i,j+1)\rightarrow(i-1,j+1)$ (see \Cref{fig:flip} and \cite{lam-lee-shimo}). For such $w$ of interest, each BPD $D$ can be obtained by a certain sequence $(F_1,...,F_t)$ of simple droops from $D_0$. Meanwhile, the corresponding $2$-path $P$ is also obtained by the same sequence $(F_1,...,F_t)$ of simple droops from $P_0$. A simple droop changes the weight to its opposite number. Therefore, 
\[\wt(D) = \wt(D_0)\cdot (-1)^t = \delta(n) \cdot \wt(P_0) \cdot (-1)^t = \delta(n) \cdot \wt(P)\]
holds for any BPD $D$ and its corresponding $2$-path $P$. 
\begin{figure}
    \centering
    \begin{tikzpicture}[scale = 0.6]
    \draw (0,0) rectangle (2,2);
    \draw (1,0)--(1,2) (0,1)--(2,1);
    \draw[path] (0.5,0.5)--(0.5,1.5)--(1.5,1.5);
    \node at (3,1) {$\longrightarrow$};
    \draw (4,0) rectangle (6,2);
    \draw (5,0)--(5,2) (4,1)--(6,1);
    \draw[path] (4.5,0.5)--(5.5,0.5)--(5.5,1.5);
    \end{tikzpicture}
\caption{A simple droop of BPD}
\label{fig:flip}
\end{figure}
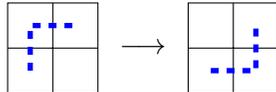
\end{proof}

Now we sum up all the BPDs of $w = w_2(2,n-2)$ to obtain 
\begin{equation}\label{equ}
\begin{aligned} \Phi_w  & = \sum_{D \in \BPD(w)} \wt(D) = \sum_{\substack{P=(P_1,P_2):S\rightarrow T\\ \text{non-intersecting}}} \delta(n) \cdot \wt(P_1)\wt(P_2)\\
& = \delta(n) \cdot \det\left(\begin{array}{cc}
\wt(S_1,T_1) & \wt(S_1,T_2)\\
\wt(S_2,T_1) & \wt(S_2,T_2)
\end{array}\right). \end{aligned}
\end{equation}

The second equality is obtained from \Cref{lemma:weight-equ}, and the last from \Cref{lemma:Lindstrom}. Therefore, $\Phi_w = \left|\S_w(1,-1,\ldots)\right|$ is exactly the absolute value of the determinant of $M_{S,T}$ as above. 

\begin{theorem}\label{lemma:matrix item}
    For double staircase partition $\Tilde{\rho}_n$, in its corresponding graph $\Tilde{G}_{\Tilde{\rho}_n}$, 
   \[M_{S,T} = \left(\begin{array}{cc}\wt(S_1,T_1) & \wt(S_1.T_2) \\ \wt(S_2,T_1) & \wt(S_2,T_2) \end{array}\right) = \begin{cases}
       \left(\begin{array}{cc}0 & -C_{k-1} \\ C_{k-1} & C_{k-1} \end{array}\right) & \text{ if } n = 2k, \\
       \left(\begin{array}{cc}C_k & * \\ 0 & -C_{k-1} \end{array}\right) & \text{ if } n = 2k+1.
   \end{cases}\]
\end{theorem}

From \Cref{lemma:matrix item} and \Cref{equ}, \Cref{thm:Catalan for double layer} is proved. Next we prove \Cref{lemma:matrix item} to complete the proof of \Cref{thm:Catalan for double layer} .

\begin{defin}
    In graph $G_{\Tilde{\rho}}$, for each end point $T_i = (i,n)$, we define \emph{zero points} of $T_i$ as those vertices $(x,y)$ of $G_{\Tilde{\rho}}$ which satisfy $\sum_{P:(x,y) \rightarrow T_i} \wt(P) = 0$, where $P$ ranges over paths from $(x,y)$ to $T_i = (i,n)$. Denote by $Z(T_i)$  the set of zero points of $T_i$. 
\end{defin}

\begin{prop}\label{prop}
    In $G_\rho$, if $(i-2,j)$ and $(i,j+2)$ are two zero points of some end point $T$, then $(i,j)$ is also a zero point of $T$. 
\end{prop}
\begin{proof}
All paths from $(i,j)$ to $T$ are divided into three types by the vertex reached after the first two steps: $(i-2,j)$, $(i-1,j+1)$, or $(i,j+2)$. Thus, \[\begin{aligned} \sum_{P:(i,j) \rightarrow T} \wt(P) = & \sum_{P:(i-2,j) \rightarrow T} \wt((i,j),(i-2,j))\wt(P)  \\ + & \sum_{P:(i-1,j+1) \rightarrow T} \wt((i,j),(i-1,j+1)) \wt(P) \\ + & \sum_{P:(i,j+2) \rightarrow T} \wt((i,j),(i,j+2)) \wt(P). \end{aligned}\]
Both the first and the third items in the right hand side are zero because $(i-2,j)$ and $(i,j+2)$ are zero points of $T$. The second item is zero since $\wt((i,j),(i-1,j+1)) = 0$. Summing up these items, we have $\sum_{P:(i,j) \rightarrow T} \wt(P) = 0$, i.e. $(i,j)$ is a zero point of $T$. 
\end{proof}

\begin{proof}[Proof of \Cref{lemma:matrix item}]
We only prove the case when $n = 2k$. The other case can be proved in the same way. 

Consider zero points of $T_1 = (1,2k)$. First, $(2,2k-1)$ is a zero point of $T_1$ because the only two paths from $(2,2k-1)$ to $(1,2k)$ have weights $1$ and $-1$, respectively. Similarly, $(2,2j-1)$ is also a zero point of $T_1$ for each $j = 1,\ldots,k$, since among all the $2(k+1-j)$ paths from $(2,2j-1)$ to $(1,2k)$, half have weight of $1$ and the other half have weight of $-1$. 

Next, $(4,2k-3)$ is a zero point. In fact, all paths from $(4,2k-3)$ to $T_1$ can be divided into two types by the vertex reached after the first two steps: $(2,2k-3)$ or $(3,2k-2)$. By the similar way of proving \Cref{prop}, we know that $(4,2k-3)$ is a zero points of $T_1$. Then using \Cref{prop} repeatedly, we obtain that $(4,2j-1)$ also belongs to $Z(T_1)$ for any $j = 1,\ldots,k-1$. 

Repeat the above procedure, we see that \[Z(T_1) = \{(i,j) \in G_\rho\ |\ i \equiv 0, j \equiv 1 \bmod 2\}.\]
See \Cref{subfig:zero t1} for illustration. 
Since $S_1 = (2k,1) \in Z(T_1)$, $\wt(S_1,T_1) = 0$. 
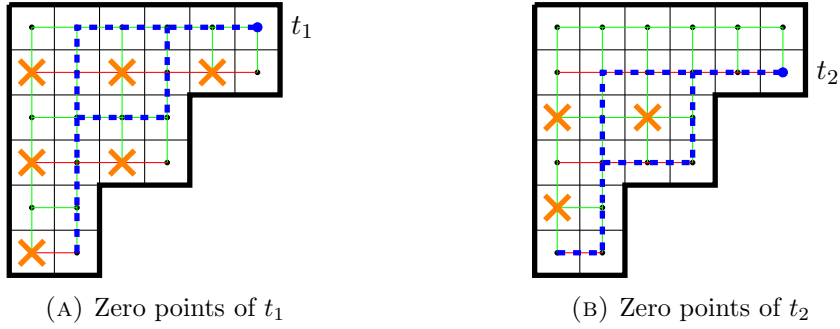
\begin{figure}
\begin{subfigure}{.45\textwidth}
    \centering
    \begin{tikzpicture}[scale = 0.6]
        \draw [line width = 2pt] (0,6)--(0,0)--(1,0)--(2,0)--(2,1)--(2,2)--(3,2)--(4,2)--(4,3)--(4,4)--(5,4)--(6,4)--(6,5)--(6,6)--(0,6);
        \draw (1,0)--(1,6);\draw (0,1)--(2,1);
        \draw (2,2)--(2,6);\draw (0,2)--(2,2);
        \draw (3,2)--(3,6);\draw (0,3)--(4,3);
        \draw (4,4)--(4,6);\draw (0,4)--(4,4);
        \draw (5,4)--(5,6);\draw (0,5)--(6,5);
        \foreach \i in {0.5,1.5,2.5,3.5,4.5,5.5} \filldraw (0.5,\i) circle (.05) (1.5,\i) circle (.05);
        \foreach \i in {2.5,3.5,4.5,5.5} \filldraw (2.5,\i) circle (.05) (3.5,\i) circle (.05);
        \foreach \i in {4.5,5.5} \filldraw (4.5,\i) circle (.05) (5.5,\i) circle (.05);
        \draw [color = green] (0.5,0.5)--(0.5,5.5) (1.5,0.5)--(1.5,5.5) (2.5,2.5)--(2.5,5.5) (3.5,2.5)--(3.5,5.5) (4.5,4.5)--(4.5,5.5) (5.5,4.5)--(5.5,5.5);
        \draw [color = green] (0.5,5.5)--(5.5,5.5) (0.5,3.5)--(3.5,3.5) (0.5,1.5)--(1.5,1.5);
        \draw [color = red] (0.5,4.5)--(5.5,4.5) (0.5,2.5)--(3.5,2.5) (0.5,0.5)--(1.5,0.5);
        \filldraw[color = blue] (5.5,5.5) circle (.1);
        \draw[path] (1.5,0.5)--(1.5,5.5)--(5.5,5.5);
        \draw[path] (1.5,3.5)--(3.5,3.5)--(3.5,5.5);
        \node (1) at (6.5,5.5) {$t_1$};
        \node [zero_node] at (4.5,4.5) {};
        \node [zero_node] at (2.5,4.5) {};
        \node [zero_node] at (0.5,4.5) {};
        \node [zero_node] at (2.5,2.5) {};
        \node [zero_node] at (0.5,2.5) {};
        \node [zero_node] at (0.5,0.5) {};
    \end{tikzpicture}
    \caption{Zero points of $t_1$}
    \label{subfig:zero t1}
\end{subfigure}
\begin{subfigure}{.45\textwidth}
    \centering
    \begin{tikzpicture}[scale = 0.6]
        \draw [line width = 2pt] (0,6)--(0,0)--(1,0)--(2,0)--(2,1)--(2,2)--(3,2)--(4,2)--(4,3)--(4,4)--(5,4)--(6,4)--(6,5)--(6,6)--(0,6);
        \draw (1,0)--(1,6);\draw (0,1)--(2,1);
        \draw (2,2)--(2,6);\draw (0,2)--(2,2);
        \draw (3,2)--(3,6);\draw (0,3)--(4,3);
        \draw (4,4)--(4,6);\draw (0,4)--(4,4);
        \draw (5,4)--(5,6);\draw (0,5)--(6,5);
        \foreach \i in {0.5,1.5,2.5,3.5,4.5,5.5} \filldraw (0.5,\i) circle (.05) (1.5,\i) circle (.05);
        \foreach \i in {2.5,3.5,4.5,5.5} \filldraw (2.5,\i) circle (.05) (3.5,\i) circle (.05);
        \foreach \i in {4.5,5.5} \filldraw (4.5,\i) circle (.05) (5.5,\i) circle (.05);
        \draw [color = green] (0.5,0.5)--(0.5,5.5) (1.5,0.5)--(1.5,5.5) (2.5,2.5)--(2.5,5.5) (3.5,2.5)--(3.5,5.5) (4.5,4.5)--(4.5,5.5) (5.5,4.5)--(5.5,5.5);
        \draw [color = green] (0.5,5.5)--(5.5,5.5) (0.5,3.5)--(3.5,3.5) (0.5,1.5)--(1.5,1.5);
        \draw [color = red] (0.5,4.5)--(5.5,4.5) (0.5,2.5)--(3.5,2.5) (0.5,0.5)--(1.5,0.5);
        \filldraw[color = blue] (5.5,4.5) circle (.1);
        \draw[path] (0.5,0.5)--(1.5,0.5)--(1.5,4.5)--(5.5,4.5);
        \draw[path] (1.5,2.5)--(3.5,2.5)--(3.5,4.5);
        \node (1) at (6.5,4.5) {$t_2$};
        \node [zero_node] at (2.5,3.5) {};
        \node [zero_node] at (0.5,3.5) {};
        \node [zero_node] at (0.5,1.5) {};
    \end{tikzpicture}
    \caption{Zero points of $t_2$}
    \label{subfig:zero t2}
\end{subfigure}
\caption{An example of zero points}
\label{fig:zero}
\end{figure}

Now we compute $\wt(S_2,T_1)$. All paths from $S_2$ to $T_1$ are divided into two types: paths without any zero points of $T_1$ and paths passing through some zero points. By an application of the inclusion-exclusion principle, we know that the latter type of paths have weights summed to zero. Thus, \[\wt(S_2,T_1) = \sum_{\substack{P:S_2 \rightarrow T_1 \\ Z(T_1) \cap P = \emptyset}} \wt(P).\]
Paths without zero points must go through the odd rows and even columns (see blue dashed lines in \Cref{subfig:zero t1}). These paths have the same weight of $1$ and the number of them equals the Catalan number $C_{k-1}$. Thus $\wt(S_2,T_1) = C_{k-1}$.

By a similar way, we also get \[Z(T_2) = \{(i,j) \in G_\rho \ |\  i \ge 3 \text{ and } i\equiv j\equiv 1\bmod 2\}.\]
See \Cref{subfig:zero t2}. Similar to the analysis above, both $\wt(S_1,T_2)$ and $\wt(S_2,T_2)$ equal exactly the weighted sum all over the paths going by those blue dashed lines shown in \Cref{subfig:zero t2}. Therefore, $\wt(S_1,T_2) = -C_{k-1}$ and $\wt(S_2,T_2) = C_{k-1}$.
\end{proof}

\subsection{Asymptotics of principal specializations for doubly layered permutations}
In this subsection, we prove \Cref{thm:asymptotic}. The strategy is to find relations between $\Phi_{w_2}$ for doubly layered permutations and $\Upsilon_w$ for layered permutations. Similar to the previous subsection, we convert the weighted sum of BPDs to the weighted sum of paths in corresponding graphs and represent the principal specializations by matrices. Surprisingly, the matrices of $\Phi_{w_2}$ and $\Upsilon_w$ have close connection.

We first review $\Upsilon_w$ for layered permutations. Keep notations as in \cite{MPP} and let \[F(m,p) := \Upsilon_{1^m \times w_0(p)},\] where $w_0(p) = (p,p-1,\ldots,1)$. A layered permutation $w(b_1,\ldots,b_k)$ equals $w_0(b_1)\times\cdots\times w_0(b_k)$.
By \Cref{equ:prop},
\[\begin{aligned}\Upsilon_{w(b_1,\ldots,b_k)} & = \Upsilon_{w(b_1,\ldots,b_{k-1})} \cdot F(b_1+\cdots+b_{k-1},b_k) \\
& = \Upsilon_{w(b_1,\ldots,b_{k-2})} \cdot F(b_1+\cdots+b_{k-2},b_{k-1}) \cdot F(b_1+\cdots+b_{k-1},b_k) \\
& = \cdots \\
& = \Upsilon_{w_0(b_1)} \cdot F(b_1,b_2) \cdot F(b_1+b_2,b_3) \cdot \cdots \cdot F(b_1+\cdots+b_{k-1},b_k) \\
& = F(b_1,b_2) \cdot F(b_1+b_2,b_3) \cdot \cdots \cdot F(b_1+\cdots+b_{k-1},b_k). 
\end{aligned}\]
From \Cref{thm:BPD wt}, $F(m,p)$ is the weighted sum of bumpless pipe dreams. A BPD $D$ of $w = 1^m \times w_0(p)$ has fixed pipe $i$ for any $i > m$. The first $m$ pipes of $D$ are pairwise non-intersecting in the partition $\rho_{m+p}^m$ (see \Cref{def:gen stair}). Therefore, each BPD corresponds to a non-intersecting $m$-paths in graph $G_{\rho_{m+p}^m}$ (see \Cref{def:stair graph}). By \Cref{lemma:Lindstrom}, 
\[F(m,p) = \det\left(\begin{array}{cccc}
\wt(s_1, t_1) & \wt(s_1, t_2) & \cdots & \wt(s_1, t_m) \\
\wt(s_2, t_1) & \wt(s_2, t_2) & \cdots & \wt(s_2, t_m) \\
\vdots & \vdots & \ddots & \vdots \\
\wt(s_m, t_1) & \wt(s_m, t_2) & \cdots & \wt(s_m, t_m)
\end{array}\right) =: \det\left(M_{ij}\right)_{m \times m},\] where $s_i,t_j \in G_{\rho^m_{m+p}}$ and $M_{ij} := \wt(s_i,t_j)$ represents the weighted sum all over paths from $s_i$ to $t_j$ in $G_{\rho^m_{m+p}}$. See \Cref{subfig:F} for an example.

Next, we consider $\Phi_w$ for doubly layered permutations. Define $\Tilde{w}_0(n)$ as \[\Tilde{w}_0(n) := (n-1,n,n-3,n-2,\ldots),\] the unique doubly layered permutation with $1$ layer. We always have $\Phi_{\Tilde{w}_0(n)} = 1$ since $S_{\Tilde{w}_0(n)} = x_1^{n-2}x_2^{n-2}x_3^{n-4}x_4^{n-4}\cdots$ by definition.

\begin{defin}
Similar to the notation of $F(m,p)$ as above, we define
\[\Tilde{F}(2m,2p+r) := \Phi_{1^{2m} \times \Tilde{w}_0(2p+r)}.\]
\end{defin}

A doubly layered permutation $w_2(2b_1,\ldots,2b_k+r)$ is $\Tilde{w}_0(2b_1) \times \cdots \times \Tilde{w}_0(2b_{k-1}) \times \Tilde{w}_0(2b_k+r)$. By \Cref{equ:prop}, 
\[\Phi_{w_2(2b_1,\ldots,2b_k+r)} = \Tilde{F}(2b_1,2b_2) \cdot \Tilde{F}(2b_1+2b_2,2b_3) \cdot \ \cdots \ \cdot \Tilde{F}(2b_1+\cdots + 2b_{k-1},2b_k+r).\]

Again from \Cref{thm:BPD wt}, $\Tilde{F}(2m,2p+r)$ is the weighted sum of bumpless pipe dreams. Same as proving \Cref{lemma:weight-equ}, as the weight setting of graph $\Tilde{G}$ (see \Cref{def:double stair graph}) matches that of BPDs under $\Phi_w$, we know that the weighted sum of BPDs and the weighted sum of non-intersecting $2m$-paths are either the same or opposite numbers of each other. Thus, $\Tilde{F}(2m,2p+r)$ equals, by \Cref{lemma:Lindstrom}, the absolute value of the matrix determinant as follows.
\[\Tilde{F}(2m,2p+r) = \det\left(\begin{array}{cccc}
\wt(S_1, T_1) & \wt(S_1, T_2) & \cdots & \wt(S_1, T_{2m}) \\
\wt(S_2, T_1) & \wt(S_2, T_2) & \cdots & \wt(S_2, T_{2m}) \\
\vdots & \vdots & \ddots & \vdots \\
\wt(S_{2m}, T_1) & \wt(S_{2m}, T_2) & \cdots & \wt(S_{2m}, T_{2m})
\end{array}\right) =: \det \left(\Tilde{M}_{ij}\right)_{2m \times 2m},\]
where $S_i,T_j \in \Tilde{G}_{\Tilde{\rho}}$, $\Tilde{\rho} = \Tilde{\rho}_{2m+2p+r}^{2m}$ and $\Tilde{M}_{ij} := \wt(S_i,T_j)$ represents the weighted sum all over paths from $S_i$ to $T_j$ in $\Tilde{G}_{\Tilde{\rho}}$. See \Cref{subfig:Tilde{F}}.

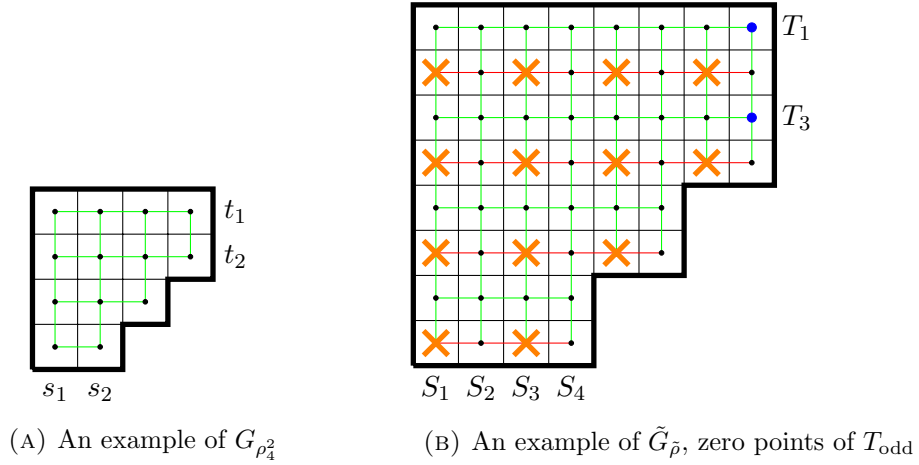
\begin{figure}[h]
    \begin{subfigure}{.45\textwidth}
    \centering
    \begin{tikzpicture}[scale = 0.6]
    \draw[line width = 2pt] (0,0)--(2,0)--(2,1)--(3,1)--(3,2)--(4,2)--(4,4)--(0,4)--(0,0);
    \draw (1,0)--(1,4) (2,1)--(2,4) (3,2)--(3,4) (0,3)--(4,3) (0,2)--(3,2) (0,1)--(2,1);
    \node at (0.5,-0.5) {$s_1$};\node at (4.5,3.5) {$t_1$};
    \node at (1.5,-0.5) {$s_2$};\node at (4.5,2.5) {$t_2$};
    \foreach \i in {0.5,1.5} \draw[color = green] (\i,0.5)--(\i,3.5);
    \foreach \i in {3.5,2.5} \draw[color = green] (0.5,\i)--(3.5,\i);
    \draw[color = green] (2.5,1.5)--(2.5,3.5) (3.5,2.5)--(3.5,3.5);
    \draw[color = green] (0.5,1.5)--(2.5,1.5) (0.5,0.5)--(1.5,0.5);
    \foreach \i in {0.5,1.5,2.5,3.5} \filldraw (0.5,\i) circle (.05) (1.5,\i) circle (.05);
    \filldraw (2.5,1.5) circle (.05) (2.5,2.5) circle (.05) (2.5,3.5) circle (.05) (3.5,2.5) circle (.05) (3.5,3.5) circle (.05);
    \end{tikzpicture}
    \caption{An example of $G_{\rho_4^2}$}
    \label{subfig:F}
\end{subfigure}
\begin{subfigure}{.45\textwidth}
    \begin{tikzpicture}[scale = 0.6]
    \draw[line width = 2pt] (0,0)--(4,0)--(4,2)--(6,2)--(6,4)--(8,4)--(8,8)--(0,8)--(0,0);
    \foreach \i in {1,2,3} \draw (\i,0)--(\i,8);
    \draw (4,2)--(4,8) (5,2)--(5,8);
    \draw (6,4)--(6,8) (7,4)--(7,8);
    \foreach \i in {5,6,7} \draw (0,\i)--(8,\i);
    \draw (0,4)--(6,4) (0,3)--(6,3);
    \draw (0,2)--(4,2) (0,1)--(4,1);
    \node at (8.5,7.5) {$T_1$};
    \node at (8.5,5.5) {$T_3$};
    \node at (0.5,-0.5) {$S_1$};
    \node at (1.5,-0.5) {$S_2$};
    \node at (2.5,-0.5) {$S_3$};
    \node at (3.5,-0.5) {$S_4$};
    \foreach \i in {0.5,1.5,2.5,3.5} \draw[color = green] (\i,0.5)--(\i,7.5);
    \foreach \i in {7.5,5.5} \draw[color = green] (0.5,\i)--(7.5,\i);
    \draw[color = green] (4.5,2.5)--(4.5,7.5)(5.5,2.5)--(5.5,7.5)(6.5,4.5)--(6.5,7.5)(7.5,4.5)--(7.5,7.5);
    \draw[color = green] (0.5,3.5)--(5.5,3.5) (0.5,1.5)--(3.5,1.5);
    \draw[color = red] (0.5,6.5)--(7.5,6.5) (0.5,4.5)--(7.5,4.5) (0.5,2.5)--(5.5,2.5) (0.5,0.5)--(3.5,0.5);
    \foreach \i in {0.5,1.5,2.5,3.5,4.5,5.5,6.5,7.5} \filldraw (0.5,\i) circle (.05) (1.5,\i) circle (.05) (2.5,\i) circle (.05) (3.5,\i) circle (.05);
    \foreach \i in {2.5,3.5,4.5,5.5,6.5,7.5} \filldraw (4.5,\i) circle (.05) (5.5,\i) circle (.05);
    \foreach \i in {4.5,5.5,6.5,7.5} \filldraw (6.5,\i) circle (.05) (7.5,\i) circle (.05);
    \filldraw[color = blue] (7.5,7.5) circle (.1) (7.5,5.5) circle (.1);
    \foreach \i in {6.5,4.5} \node[zero_node] at (6.5,\i) {};
    \foreach \i in {6.5,4.5,2.5} \node[zero_node] at (4.5,\i) {};
    \foreach \i in {6.5,4.5,2.5,0.5} \node[zero_node] at (2.5,\i) {};
    \foreach \i in {6.5,4.5,2.5,0.5} \node[zero_node] at (0.5,\i) {};
    \end{tikzpicture}
    \caption{An example of $\Tilde{G}_{\Tilde{\rho}}$, zero points of $T_{\text{odd}}$}
    \label{subfig:Tilde{F}}
\end{subfigure}
\caption{Examples for $G_{\rho}$ and $\Tilde{G}_{\Tilde{\rho}}$}
\label{fig:F&Tilde{F}}
\end{figure}

\begin{theorem}\label{thm:F}
    $\Tilde{F}(2m,2p+r) = \begin{cases}F(m,p)^2 & \text{ if } r = 0, \\ F(m,p)\cdot F(m,p+1) & \text{ if } r = 1. \end{cases}$
\end{theorem}
It is a generalization of \Cref{thm:Catalan for double layer}. In particular, when $m = 1$, \Cref{thm:F} becomes \Cref{thm:Catalan for double layer}. 

\begin{proof}
     Case $r=0$. As the notation above, denote by $M = \left(M_{ij}\right)$ the matrix of $F(m,p)$ and by $\Tilde{M} = \left(\Tilde{M}_{ij}\right)$ the matrix of $\Tilde{F}(2m,2p)$. We will write $\Tilde{M}_{ij}$'s in terms of $M_{ij}$'s. 
    Consider all the zero points of each end point $T$. Recall
    \[\begin{aligned}Z(T_{2k-1}) &= \{(i,j) \in \Tilde{G} \ |\ i > 2k-1 \text{ and } i \equiv 0, j \equiv 1 \bmod 2\},\\ Z(T_{2k}) &= \{(i,j) \in \Tilde{G} \ |\ i >2k \text{ and } i \equiv j \equiv 1 \bmod 2\}, \end{aligned}\]
    for every $1 \le k \le m$. See \Cref{subfig:Tilde{F}} for an example of $Z(T_{\text{odd}})$. Moreover, by the inclusion-exclusion principle, we know that the weighted sum $\wt(S_i,T_j)$ over the paths from $S_i$ to $T_j$ equals those without passing through any points in $Z(T_j)$. That is 
    \[\wt(S_i,T_j) = \sum_{\substack{P:S_i \rightarrow T_j \\ Z(T_j) \cap P = \emptyset}}\wt(P).\]
    Since $S_{2i-1} = (2m+2p,2i-1) \in Z(T_{2j-1})$, we have $\Tilde{M}_{2i-1,2j-1} = \wt(S_{2i-1},T_{2j-1}) = 0$. For $\Tilde{M}_{2i,2j-1}$, banning all the rows and columns which intersect with zero points $Z(T_{2j-1})$, we are left with odd rows and even columns, which form exactly the graph $G = G_{\rho_{m+p}^m}$. Thus, $\Tilde{M}_{2i,2j-1} = \wt(S_{2i},T_{2j-1}) = M_{ij}$. 
    Similarly, for $\Tilde{M}_{2i-1,2j}$ and $\Tilde{M}_{2i,2j}$, we only need to consider the rows and columns avoiding any zero points $Z(2j)$, which are all the even rows and even columns. They also form the same shape as the graph $G_{\rho_{m+p}^m}$, with edges in rows having weight $-1$ instead of $1$. Thus,
    \[\begin{aligned}\Tilde{M}_{2i-1,2j} &= \wt(S_{2i-1},T_{2j}) = (-1)^{2m+2p-2i+1} \cdot M_{ij} = -M_{ij},\\
    \Tilde{M}_{2i,2j} &= \wt(S_{2i},T_{2j}) = (-1)^{2m+2p-2i} \cdot M_{ij} = M_{ij} \end{aligned}.\]
    Therefore, for all $1 \le i,j \le m$, we obtain
    \[\left(\begin{array}{cc}\Tilde{M}_{2i{-}1,2j{-}1} & \Tilde{M}_{2i{-}1,2j}\\ \Tilde{M}_{2i,2j{-}1} & \Tilde{M}_{2i,2j}\end{array}\right) = \left(\begin{array}{cc} 0 & -M_{ij}\\ M_{ij} & M_{ij}\end{array}\right).\]
    
    Now perform elementary operations on matrix $\Tilde{M}$. First, add row $2i-1$ to row $2i$ for every $1 \le i \le m$. Next, move all the even rows to the top and move all the odd columns to the left. Then, matrix $\Tilde{M}$ becomes $\left(\begin{array}{cc} M &0 \\ 0 &-M \end{array}\right)$. 
    Therefore,
    \[\Tilde{F}(2m,2p) = \left|\det(\Tilde{M})\right| = \left|\det\left(M\right)\right| \cdot \left|\det(M)\right| = F(m,p)^2. \]
    
    Case $r=1$. The proof is similar as above, but the details are more complicated. Denote by $M = \left(M_{ij}\right)$, $M^{\prime} = \left(M^{\prime}_{ij}\right)$ and $\Tilde{M} = \left(\Tilde{M}_{ij}\right)$ the matrix of $F(m,p)$, $F(m,p+1)$ and $\Tilde{F}(2m,2p+1)$, respectively. We also use the zero points to simplify $\Tilde{G}$ so that we can use items $M_{ij}$ and $M^{\prime}_{ij}$ to represent $\Tilde{M}$. Finally, for $2 \le i \le m$, $1 \le j \le m$,
    \[\left(\begin{array}{cc}\Tilde{M}_{2i{-}1,2j{-}1} & \Tilde{M}_{2i{-}1,2j}\\ \Tilde{M}_{2i,2j{-}1} & \Tilde{M}_{2i,2j}\end{array}\right) = \left(\begin{array}{cc} M^{\prime}_{i,j} & M_{i-1,j}{-}M_{m,j}\\ M^{\prime}_{i+1,j} & {-}M_{m,j}\end{array}\right),\]
    where we define $M^{\prime}_{m+1,j} = 0$.
    In particular, when $i = 1$, we can only deduce that $\left(\begin{array}{cc}\Tilde{M}_{1,2j{-}1} & \Tilde{M}_{1,2j}\\ \Tilde{M}_{2,2j{-}1} & \Tilde{M}_{2,2j} \end{array}\right) = \left(\begin{array}{cc} M^{\prime}_{1,j} & *\\ M^{\prime}_{2,j} & {-}M_{m,j}\end{array}\right)$ for $1 \le j \le m$. Though $\Tilde{M}_{1,2j}$ is hard to determine, we do not need its exact value.

    Now perform elementary operations on matrix $\Tilde{M}$. First, subtract the last row from each row except itself. Next, subtract row $2i$ from row $2i+1$ for every $1 \le i \le m-1$. Finally, move rows $1,2,4,\ldots,2m-2$ to the top and move all the odd columns to the left. Then the matrix $\Tilde{M}$ becomes
    $\left(\begin{array}{cc} M^{\prime} & * \\ 0 & M \end{array}\right)$.
    Therefore,
    \[\Tilde{F}(2m,2p+1) = \left|\det(\Tilde{M})\right| = \left|\det(M^{\prime})\right| \cdot \left|\det(M)\right| = F(m,p+1) \cdot F(m,p). \]
\end{proof}

\begin{cor}\label{cor:doubly}
For a doubly layered permutation $w_2(2b_1,\ldots,2b_k+r)$,
\[\Phi_{w_2(2b_1,\ldots,2b_k+r)} = \begin{cases} (\Upsilon_{w(b_1,\ldots,b_k)})^2 & \text{ if } r = 0, \\ \Upsilon_{w(b_1,\ldots,b_k)}\cdot \Upsilon_{w(b_1,\ldots,b_k+1)} & \text{ if } r = 1. \end{cases}\]
\end{cor}

\begin{proof}
Applying \Cref{thm:F} repeatedly, we can represent the principal specializations $\Phi_w$ for every doubly layered permutations $w$ by $\Upsilon_w$ for some layered permutations. As a matter of fact, for each $w = w_2(2b_1,\ldots,2b_k+r)$,
\[\begin{aligned}
    \Phi_w = & \Tilde{F}(2b_1,2b_2) \cdot \cdots \cdot \Tilde{F}(2b_1+\cdots+2b_{k-1},2b_k+r) \\
    = & F(b_1,b_2)^2 \cdot \cdots \cdot F(b_1+\cdot+b_{k-2},b_{k-1})^2 \\ & \cdot F(b_1+\cdots+b_{k-1},b_k) \cdot F(b_1+\cdots+b_{k-1},b_k+r) \\
    = & \Upsilon_{w(b_1,\ldots,b_{k-1},b_k)} \cdot \Upsilon_{w(b_1,\ldots,b_{k-1},b_k+r)}.
\end{aligned}\]
\end{proof}

Now we are ready to prove one of our main results.
\begin{proof}[Proof of \Cref{thm:asymptotic}]
By \Cref{cor:doubly}, 
\[\Tilde{v}(n) = \begin{cases}v(k)^2 & \text{ if } n = 2k, \\ v(k)v(k+1) & \text{ if } n = 2k+1. \end{cases}\]
Therefore, there is a limit
\[\lim_{n \to \infty} \frac{\log_2 \Tilde{v}(n)}{n^2} = \lim_{k \to \infty} \frac{\log_2 v(k)^2}{4k^2} = \frac{1}{2}\lim_{k \to \infty} \frac{\log_2 v(k)}{k^2} = \approx 0.1466.\]
\end{proof}

\section{Multi-layered Permutations}\label{sec:multi}
Results of principal specialization $\Phi_w$ for doubly layered permutations in the previous section can be generalized to $\S_w(1,q,q^2,\ldots)$ at roots of unity for multi-layered permutations. For the sake of a clean exposition, we choose to deal with the case of $q=-1$ in details in the previous section and outline the necessary crucial steps in the current section for multi-layered permutations. 

To be specific, fix any $k$-th root of unity $\zeta=e^{\frac{2\pi j}{k}}$ where $\mathrm{gcd}(j,k)=1$ and consider the principal specialization 
\[\Phi^k_w :=\left|\S_w(1,q,q^2,\ldots)\right|_{q=\zeta}.\] 
Define multi-layered permutations as follows.

\begin{defin}
    A \emph{$k$-multi-layered permutation} in $S_n$ is defined as
    \[\begin{aligned}w_k(kb_1,\ldots,kb_t+r) := ( & \underbrace{kb_1-k+1, \ldots,kb_1,\ldots,1, \ldots, k}_{\text{ the first } k{-}\text{multi layer}}, \\ & \underbrace{kb_1+kb_2-k+1,\ldots,kb_1+kb_2,\ldots,kb_1+1,\ldots,kb_1+k}_{\text{ the second } k{-}\text{multi layer}}, \\ & \ldots, \\ &\underbrace{n-k+1,\ldots,n,\ldots,n-kb_t-r+1,\ldots,n-kb_t}_{\text{ the } t{-}\text{th } k{-}\text{multi layer}}),\end{aligned}\]
    where $n = kb_1+ \cdots +kb_t+r$ and r is the remainder of $n \bmod k$. In particular, when $k = 1$ and $2$, it becomes layered permutation and doubly layered permutation, respectively. 
\end{defin}

Denote by $k\mathcal{L}_n$ the set of $k$-multi-layered permutations in $S_n$. Let $v_k(n)$ be the largest principal specializations at $q=\zeta$ in $k\mathcal{L}_n$ \[v_k(n) := \max_{w \in k\mathcal{L}_n} \Phi^k_w.\] Analogous to \Cref{thm:Catalan for double layer}, we have the following theorem for multi-layered permutations.

\begin{theorem}\label{thm:multi-Catalan}
    Given $n = kn_0+r$ and $w = w_k(k,k(n_0-1)+r)$ with $0\leq r<k$,
    \[\Phi^k_w = (C_{n_0-1})^{k-r}\cdot (C_{n_0})^r,\]
    where $C_n$ represents the $n$-th Catalan number.
\end{theorem}

Next, similarly as above, define $w^k_0(n) := (n-k+1,\ldots,n,n-2k+1,\ldots,n-k,\ldots)$ and $F^k(km,kp+r) := \Phi^k_{1^{km}\times w^k_0(kp+r)}$, where $0 \le r < k$.
We always have $\Phi^k_{w^k_0(n)} = 1$ and 
\[\Phi^k_{w_k(kb_1,\ldots,kb_t+r)} = F^k(kb_1,kb_2) \cdot F^k(kb_1+kb_2,kb_3) \cdot \ \cdots \ \cdot F^k(kb_1+\cdots + kb_{k-1},kb_k+r).\]
In particular, when $k = 2$, $w^k_0(n)$ is $\Tilde{w}_0(n)$ and $F^k(2m,2p+r)$ is $\Tilde{F}(2m,2p+r)$.

Analogous to \Cref{thm:F} and \Cref{cor:doubly}, we have the following results for multi-layered permutations under principal specializations $\Phi_w^k$. 

\begin{theorem}
    Given $0 \le r < k$, we have
    \[F^k(km,kp+r) = \left(F(m,p)\right)^{k-r}\cdot \left(F(m,p+1)\right)^r.\]
    In particular, when $m = 1$, we recover the equation in \Cref{thm:multi-Catalan}.
\end{theorem}

\begin{cor}
    For a $k$-multi-layered permutation $w_k(kb_1,\ldots,kb_t,r)$, we have
    \[\Phi^k_{w_k(kb_1,\ldots,kb_t+r)} = (\Upsilon_{w(b_1,\ldots,b_t)})^{k-r} \cdot (\Upsilon_{w(b_1,\ldots,b_t+1)})^r,\]
    where $w(b_1,\ldots,b_t)$ and $w(b_1,\ldots,b_t+1)$ are two layered permutations.
\end{cor}

All the statements can be verified via the same way as doubly layered permutations. The method is to construct a graph $G$ corresponding to a certain staircase partition $\rho$ and to calculate the Schubert polynomial by transferring the weighted sum of bumpless pipe dreams in $\rho$ to the weighted sum of paths in $G$. Note that each horizontal edge on row $j$ has weight $\zeta^{j-1}$.
Therefore, we obtain the generalized version of our main result \Cref{thm:asymptotic}.

\begin{theorem}
There is a limit
\[\lim_{n \to \infty} \frac{\log_2 v_k(n)}{n^2} = \frac{1}{k}\lim_{n \to \infty} \frac{\log_2 v(n)}{n^2} = \frac{\gamma}{k \log 2} \approx \frac{0.2932}{k}.\]
\end{theorem}

For the largest principal specialization at $q = 1$, Stanley gave the upper bound for $u(n)$ based on the Cauchy identity for Schubert polynomials. 
\begin{theorem}[\cite{stanley}]
    $\limsup_{n \to \infty} \frac{\log_2 u(n)}{n^2} \le \frac12$.
\end{theorem}
Let $u_k(n)$ be the largest principal specialization at $k$-th unit root \[u_k(n) := \max_{w \in S_n} \Phi^k_w.\] It is easy to see that, given $k$, $\Phi_w^k \le \Upsilon_w$ always holds for each $w \in S_n$. Thus, \[u_k(n) \le u(n),\] for any $k \ge 1$. The upper bound for $u(n)$ is also an upper bound for $u_k(n)$. Moreover, $v_k(n)$ naturally forms a lower bound for $u_k(n)$. Therefore, we have
\begin{theorem}
    For any $k \ge 1$ fixed, 
    \[\frac{\gamma}{k\log 2} \le \liminf_{n \to \infty} \frac{\log_2 u_k(n)}{n^2} \le \limsup_{n \to \infty} \frac{\log_2 u_k(n)}{n^2} \le \frac12.\]
\end{theorem}
This theorem gives the first order of asymptotics for the largest principal specialization $u_k(n)$ at $q=\zeta$. A natural further question is how to narrow down the upper bound for $u_k(n)$. 
At the end of this paper, we make the following conjecture.
\begin{conj}
    Given $k \in \mathbb{Z}_{>0}$, for every $n$, all permutations reaching the maximum $u_k(n)$ under principal specialization $\S_w(1,q,q^2,\ldots)$ at $q=\zeta$ are $k$-multi-layered permutations. In other words, $u_k(n) = v_k(n)$. 
    Thus, there is a limit \[\lim_{n \to \infty} \frac{\log_2 u_k(n)}{n^2} = \frac{\gamma}{k\log 2}.\]
\end{conj}

\section*{Acknowledgements}
The author thanks Prof. Yibo Gao for proposing this problem. 

\bibliographystyle{plain}
\bibliography{ref.bib}

\begin{thebibliography}{10}

\bibitem{Billey}
Sara~C. Billey, Alexander~E. Holroyd, and Benjamin~J. Young.
\newblock A bijective proof of {M}acdonald's reduced word formula.
\newblock {\em Algebr. Comb.}, 2(2):217--248, 2019.

\bibitem{dennin2022pattern}
Hugh Dennin.
\newblock Pattern bounds for principal specializations of $\beta$-grothendieck
  polynomials.
\newblock {\em arXiv preprint arXiv:2206.10017}, 2022.

\bibitem{Fomin}
Sergey Fomin and Richard~P. Stanley.
\newblock Schubert polynomials and the nil-{C}oxeter algebra.
\newblock {\em Adv. Math.}, 103(2):196--207, 1994.

\bibitem{gao2021principal}
Yibo Gao.
\newblock Principal specializations of {S}chubert polynomials and pattern
  containment.
\newblock {\em European J. Combin.}, 94:Paper No. 103291, 12, 2021.

\bibitem{lam-lee-shimo}
Thomas Lam, Seung~Jin Lee, and Mark Shimozono.
\newblock Back stable {S}chubert calculus.
\newblock {\em Compos. Math.}, 157(5):883--962, 2021.

\bibitem{SchubertNotes}
I.G. Macdonald.
\newblock {\em Notes on Schubert polynomials}, volume~6 of {\em Publ.LaCIM}.
\newblock Universit\'e de Qu\'ebec \`a Montr\'eal, Montr\'eal, Canada, 1991.

\bibitem{Merzon}
Grigory Merzon and Evgeny Smirnov.
\newblock Determinantal identities for flagged {S}chur and {S}chubert
  polynomials.
\newblock {\em Eur. J. Math.}, 2(1):227--245, 2016.

\bibitem{meszaros2021inclusion}
Karola M{\'e}sz{\'a}ros and Arthur Tanjaya.
\newblock Inclusion-exclusion on schubert polynomials.
\newblock {\em arXiv preprint arXiv:2102.11179}, 2021.

\bibitem{MPP}
Alejandro~H. Morales, Igor Pak, and Greta Panova.
\newblock Asymptotics of principal evaluations of {S}chubert polynomials for
  layered permutations.
\newblock {\em Proc. Amer. Math. Soc.}, 147(4):1377--1389, 2019.

\bibitem{ec1}
Richard~P. Stanley.
\newblock {\em Enumerative combinatorics. {V}olume 1}, volume~49 of {\em
  Cambridge Studies in Advanced Mathematics}.
\newblock Cambridge University Press, Cambridge, second edition, 2012.

\bibitem{stanley}
Richard~P. Stanley.
\newblock Some schubert shenanigans, 2017.

\end{thebibliography}
\end{document}